\documentclass[11pt]{amsart}
\usepackage{amsmath,amssymb,amsthm,fancyhdr}
\usepackage{hyperref,color}
\usepackage[letterpaper,margin=1in]{geometry}
\usepackage[dvipsnames]{xcolor}

\usepackage[hang,flushmargin]{footmisc}

\newtheorem{theorem}{Theorem}[section]

\newtheorem{prop}[theorem]{Proposition}

\newtheorem{corollary}[theorem]{Corollary}
\newtheorem*{problem}{Problem}

\usepackage[english]{babel}

\DeclareMathOperator{\Ap}{Ap}

\usepackage[T1]{fontenc} 
\usepackage{multirow}

\def\N{\mathbb{N}}

\title{On the Genus of a quotient of a numerical semigroup}
\makeatletter
\let\thetitle\@title        
\let\theauthor\@author      

\date{}

\author[Adeniran]{Ayomikun Adeniran}
\address{Department of Mathematics, Texas A\&M University, United States}
\email{\textcolor{blue}{\href{mailto:ayoijeng88@tamu.edu}{ayoijeng88@tamu.edu}}}

\author[Butler]{Steve Butler}
\address{Department of Mathematics, Iowa State University, United States}
\email{\textcolor{blue}{\href{mailto:butler@iastate.edu}{butler@iastate.edu}}}

\author[Defant]{Colin Defant}
\address{Department of Mathematics, Princeton University, United States}
\email{\textcolor{blue}{\href{mailto:cdefant@princeton.edu}{cdefant@princeton.edu}}}

\author[Gao]{Yibo Gao}
\address{Department of Mathematics, MIT, United States}
\email{\textcolor{blue}{\href{mailto:gaoyibo@mit.edu}{gaoyibo@mit.edu}}}

\author[Harris]{Pamela E. Harris}
\address{Department of Mathematics and Statistics, Williams College, United States}
\email{\textcolor{blue}{\href{mailto:peh2@williams.edu}{peh2@williams.edu}}}

\author[Hettle]{Cyrus Hettle}
\address{School of Mathematics, Georgia Institute of Technology}
\email{\textcolor{blue}{\href{mailto:chettle@gatech.edu}{chettle@gatech.edu}}}

\author[Liang]{Qingzhong Liang}
\address{Department of Mathematics, Duke University}
\email{\textcolor{blue}{\href{mailto:qingzhong.liang@duke.edu}{qingzhong.liang@duke.edu}}}

\author[Nam]{Hayan Nam}
\address{Department of Mathematics, University of California, Irvine }
\email{\textcolor{blue}{\href{mailto:hayann@uci.edu}{hayann@uci.edu}}}

\author[Volk]{Adam Volk}
\address{Department of Mathematics, University of Nebraska-Lincoln}
\email{\textcolor{blue}{\href{mailto:avolk@huskers.unl.edu}{avolk@huskers.unl.edu}}}
\begin{document}
\maketitle

\begin{abstract}
We find a relation between the genus of a quotient of a numerical semigroup $S$ and the genus of $S$ itself. We use this identity to compute the genus of a quotient of $S$ when $S$ has embedding dimension $2$. We also exhibit identities relating the Frobenius numbers and the genus of quotients of numerical semigroups that are generated by certain types of arithmetic progressions.
\end{abstract}

\section{Introduction}

Throughout this paper, let $\N$ denote the set of nonnegative integers. A \emph{numerical semigroup} is a subset of $\N$ that is closed under addition, contains $0$, and has finite complement in $\N$. Given positive integers $a_1, \ldots, a_n$ satisfying $\gcd(a_1,\ldots,a_n)=1$, we write \[\langle a_1\ldots,a_n\rangle=\{c_1a_1+\cdots+c_na_n:c_1,c_2,\ldots,c_n\in\N\}.\] The set $\langle a_1\ldots,a_n\rangle$ is a numerical semigroup called the numerical semigroup \emph{generated by the set} $\{a_1,\ldots,a_n\}$. It is well known that every numerical semigroup is of this form~\cite[Theorem 2.7]{del}. That is, every numerical semigroup is generated by a finite set of positive integers. Furthermore, every numerical semigroup has a unique set of generators that is minimal in the sense that no proper subset of the generating set generates the same numerical semigroup.

When studying a numerical semigroup $S$, it is useful to consider the Hilbert series 
\begin{equation}\label{eq:HSx}
H_S(x)=\sum_{s\in S}x^s
\end{equation}
and the semigroup polynomial
\begin{equation}\label{eq:PSx}
P_S(x)=(1-x)H_S(x)=1-(1-x)\sum\limits_{s \not\in S}x^s.
\end{equation}
Note that $P_S(x)$ is a polynomial because $S$ has finite complement in $\N$. 

There are several fundamental invariants of a numerical semigroup $S$. The \emph{Frobenius number} and \emph{genus} of $S$, denoted $F(S)$ and $g(S)$, respectively, are defined by \[F(S)=\max(\N\setminus S)\hspace{.5cm}\text{and}\hspace{.5cm}g(S)=|\N\setminus S|.\] The size of the unique minimal generating set of $S$ is called the \emph{embedding dimension} of $S$. We say $S$ is $d$-\emph{symmetric} if $n \notin S$ implies $F(S)-n \in S$ whenever $n$ is a positive multiple of $d$.  A $1$-symmetric numerical semigroup is simply called \emph{symmetric}.

It is difficult to give general formulas for the invariants of numerical semigroups. 
However, some special types of numerical semigroups have received a large amount of attention, and the invariants of these numerical semigroups are often well-understood. 
This includes numerical semigroups with small embedding dimensions \cite{Davison, Mosc1, Syl}, numerical semigroups generated by arithmetic progressions \cite{MSS,Num}, and $d$-symmetric numerical semigroups \cite{Stra}. For example, the following very fundamental result is due to Sylvester. 

\begin{theorem}[\cite{Syl}]\label{SylvesterTheorem}
If $a$ and $b$ are relatively prime positive integers, then \[F(\langle a,b\rangle)=ab-a-b\quad\text{and}\quad g(\langle a,b\rangle)=(a-1)(b-1)/2.\]
\end{theorem}

The \emph{quotient} of a numerical semigroup $S$ by a positive integer $d$ is the set \[S/d=\{x \in \N:dx \in S\}.\]
It is not difficult to see that $S/d$ is also a numerical semigroup. Thus, we define the Frobenius number $F(S/d)$ and the genus $g(S/d)$ as above. Recently, there has been a flurry of investigation of invariants of quotients of numerical semigroups by positive integers~\cite{Ze,Mos1,Mosc2,RosGar,RU}. 
For example, the following result will prove particularly useful for our purposes.  
\begin{theorem}[\cite{Stra}]
For a $d$-symmetric numerical semigroup $S$, we have
\[F(S/d) = \frac{F(S)-x}{d},\] where $x$ is the smallest positive element of $S$ satisfying $x \equiv F(S) \pmod d$.
\end{theorem}

One motivation for studying quotients of numerical semigroups comes from the study of \emph{proportionally modular Diophantine inequalities}, which are Diophantine inequalities of the form $ax\pmod b\leq cx$ for some fixed positive integers $a,b,c$. It turns out that the set of nonnegative integer solutions to a proportionally modular Diophantine inequality form a numerical semigroup; a numerical semigroup obtained in this way is called a \emph{proportionally modular} numerical semigroup. Robles-P\'erez and Rosales \cite{RR} have shown that a numerical semigroup is proportionally modular if and only if it is of the form $\langle a,a+1\rangle/d$ for some positive integers $a$ and $d$. Furthermore, Delgado, Garc\'{i}a-S\'{a}nchez, and Rosales \cite{del} have remarked that there is no known example of a numerical semigroup that is not of the form $\langle a,b,c\rangle/d$. It is natural to study quotients of the special numerical semigroups whose invariants are already well understood. The current paper is focused on proving the following theorem and then applying it to gain information about quotients of numerical semigroups of the form $\langle a,b\rangle$, $\langle a,a+k,a+2k\rangle$, or $\langle a,a+k,\ldots,a+(a-1)k\rangle$. 

\begin{theorem}\label{main}
Let $S$ be a numerical semigroup, and let $d$ be a positive integer. We have
\begin{align}
g(S/d)&=\frac{1}{d}\left[g(S)+\frac{d-1}{2}-\sum\limits_{i=1}^{d-1}H_S({\zeta_d}^i)\right],
\end{align}
where $\zeta_d$ is a primitive $d^\text{th}$ root of unity. 
\end{theorem}

Recall that $q(x)=c_d(x)x^k+c_{d-1}(x)x^{k-1}+\cdots+c_0(x)$ is a quasipolynomial of degree $k$ in the variable $x$ if the coefficients $c_i(x)$ are periodic functions of $x$. Strazzanti proved that $F\left(\frac{\langle a, b \rangle}{2}\right)$ and $F\left(\frac{\langle a, a+1 \rangle}{5}\right)$ are quasipolynomials in $a$ of degree $2$ (with periods $2$ and $5$, respectively) \cite{Stra}. By specializing Theorem \ref{main} to the case $S=\langle a, a+k\rangle$, where $k$ is a fixed positive integer that is coprime to $a$, we show that the genus of $S/d$ is a quasipolynomial in $a$ of degree $2$ (see Corollary \ref{corofmain}). This result generalizes the formulas for $g({\langle a, b \rangle}/2)$ (when $\gcd(a,b)=1$) given in \cite{Ze} to formulas for $g({\langle a, b \rangle}/d)$ for all $d\geq 2$. The latter result answers an open problem listed in \cite{del}.

We prove Theorem \ref{main} in the next section.
Section \ref{sec:arithmetic} contains some results on invariants of quotients of numerical semigroups generated by certain arithmetic progressions, and we end the section by providing an open problem for the future study.

\section{The genus of a quotient of a numerical semigroup }\label{sec:generalresult}

We begin by proving Theorem \ref{main}. Fix a positive integer $d$, and let $A_i=|\{s\in\N\setminus S:s \equiv i \pmod d \}|$. In particular, $A_0=g(S/d)$.
Letting $\zeta_d$ denote a primitive $d^\text{th}$ root of unity, we have
\begin{align*}
P_S(\zeta_d) &= 1-(1-\zeta_d)(A_0+\zeta_d A_1+{\zeta_d}^2A_2+\cdots+{\zeta_d}^{d-1}A_{d-1})\\
P_S({\zeta_d}^2) &= 1-(1-{\zeta_d}^2)(A_0+{\zeta_d}^2 A_1+{\zeta_d}^4A_2+\cdots+{\zeta_d}^{2(d-1)}A_{d-1})\\
 &\;\;\vdots \\
P_S({\zeta_d}^{d-1}) &= 1-(1-{\zeta_d}^{d-1})(A_0+{\zeta_d}^{d-1} A_1+{\zeta_d}^{2(d-1)}A_2+\cdots+{\zeta_d}^{(d-1)^2}A_{d-1}).
\end{align*}
We also have the identity $g(S)=A_0+A_1+\cdots+A_{d-1}$.
By rewriting the above equations, we get
\begin{align*}
 A_0+\zeta_d A_1+{\zeta_d}^2A_2+\cdots+{\zeta_d}^{d-1}A_{d-1} &= \frac{1-P_S(\zeta_d)}{1-\zeta_d}\\
A_0+\zeta_d^2 A_1+{\zeta_d}^4A_2+\cdots+{\zeta_d}^{2(d-1)}A_{d-1} &= \frac{1-P_S(\zeta_d^2)}{1-\zeta_d^2}\\
 &\;\;\vdots \\
A_0+\zeta_d^{d-1} A_1+{\zeta_d}^{2(d-1)}A_2+\cdots+{\zeta_d}^{(d-1)^2}A_{d-1} &= \frac{1-P_S(\zeta_d^{d-1})}{1-\zeta_d^{d-1}}\\
A_0+A_1+\cdots +A_{d-1}&=g(S).
\end{align*}
Adding all of these together and using the fact that $1+\zeta_d+\zeta_d^2+\cdots+\zeta_d^{d-1}=0$ gives
	\begin{equation}\label{Eq2}
	A_0=\frac{1}{d}\left(g(S)+\sum\limits_{n=1}^{d-1}\frac{1-P_S({\zeta_d}^n)}{1-{\zeta_d}^n}\right)=\frac{1}{d}\left(g(S)+\sum\limits_{n=1}^{d-1}\frac{1}{1-{\zeta_d}^n}-\sum\limits_{n=1}^{d-1}H_S({\zeta_d}^n)\right).
	\end{equation}
	
The desired result now follows from the identity 
\begin{equation}\label{Eq1}
\sum\limits_{n=1}^{d-1}\frac{1}{1-{\zeta_d}^n}=\frac{d-1}{2}.
\end{equation}
To prove \eqref{Eq1}, note that for $1 \leq k \leq d-1$, we have
\[\frac{1}{1-{\zeta_d}^k}+\frac{1}{1-{\zeta_d}^{d-k}}=\frac{1}{1-{\zeta_d}^k}+\frac{1}{1-{\zeta_d}^{-k}}=1.\]
This immediately implies \eqref{Eq1} when $d$ is odd.
If $d$ is even, then \[\sum_{n=1}^{d-1}\frac{1}{1-\zeta_d^n}=\frac{d-2}{2}+\frac{1}{1-{\zeta_d}^{d/2}}=\frac{d-1}{2}\] because ${\zeta_d}^{d/2}=-1$. This completes the proof of Theorem \ref{main}.

\begin{corollary}\label{cor:S1}
Let $a,b,$ and $d$ be relatively prime positive integers, and let $a^*$ be the unique integer satisfying $aa^*\equiv 1\pmod d$ and $1 \leq a^* \leq d-1$. We have \[
g\left(\langle a,b \rangle/d \right)=\frac{(a-1)(b+d-a^*ab)}{2d}+\frac{1}{2}\left \lfloor \frac{a-1}{d}\right \rfloor\left(a^*b\left \lfloor \frac{a-1}{d}\right \rfloor+a^*b-2\right)+\sum\limits_{\substack{j=1\\ d \nmid j }}^{a-1}\left\lfloor \frac{a^*bj}{d}\right\rfloor.
\]
\end{corollary}

\begin{proof}
Let $S=\langle a, b \rangle$. It is straightforward to check that $H_S(x)=\frac{1-x^{ab}}{(1-x^a)(1-x^b)}$. Now,
{
\allowdisplaybreaks
\begin{align*}
\sum\limits_{i=1}^{d-1}H_S({\zeta_d}^i)&=\sum\limits_{i=1}^{d-1}\frac{1+{\zeta_d}^{bi}+{\zeta_d}^{2bi}+\cdots+{\zeta_d}^{(a-1)bi}}{1-{\zeta_d}^{ai}}\\
&=\sum\limits_{i=1}^{d-1}\sum\limits_{j=0}^{a-1}\frac{{\zeta_d}^{bji}}{1-{\zeta_d}^{ai}}\\
&=\sum\limits_{i=1}^{d-1}\sum\limits_{j=0}^{a-1}\frac{{\zeta_d}^{a^*bji}}{1-{\zeta_d}^{i}}\\
&=\sum\limits_{i=1}^{d-1}\frac{1}{1-{\zeta_d}^i}+\sum\limits_{i=1}^{d-1}\sum\limits_{j=1}^{a-1}\frac{{\zeta_d}^{a^*bji}}{1-{\zeta_d}^{i}}\\
&=\frac{d-1}{2}+\sum\limits_{i=1}^{d-1}\sum\limits_{j=1}^{a-1}\frac{{\zeta_d}^{a^*bji}-1+1}{1-{\zeta_d}^{i}}\\
&=\frac{d-1}{2}+\sum\limits_{i=1}^{d-1}\sum\limits_{j=1}^{a-1}\left(-(1+\zeta_d^i+\zeta_d^{2i}+\cdots+\zeta_d^{(a^*bj-1)i})+\frac{1}{1-\zeta_d^i}\right)\\
&=\frac{a(d-1)}{2}-\sum\limits_{j=1}^{a-1}\sum\limits_{i=1}^{d-1}\left(1+\zeta_d^i+\zeta_d^{2i}+\cdots+\zeta_d^{(a^*bj-1)i}\right)\\
&= \frac{a(d-1)}{2}-\sum\limits_{\substack{j=1\\d \nmid j}}^{a-1}\sum\limits_{i=1}^{d-1}\left(\zeta_d^{\left\lfloor \frac{a^*bj}{d}\right\rfloor di}+\cdots+\zeta_d^{(a^*bj-1)i}\right)\\
&= \frac{a(d-1)}{2}-\sum\limits_{\substack{j=1\\d \nmid j}}^{a-1}\left\{(d-1)+(-1)\left(a^*bj-1-\left\lfloor \frac{a^*bj}{d}\right\rfloor d\right)\right\}
\end{align*}
}
The third equation comes from by setting new $\zeta_d$ as $\zeta_d^a$. Note that we have made use of the identity in \eqref{Eq1} and the identity $1+\zeta_d^i+\zeta_d^{2i}+\cdots+\zeta_d^{(d-1)i}=0$ for $1 \leq i \leq d-1$. The desired result is now an immediate consequence of Theorem \ref{main}. 
\end{proof}

We now specialize to numerical semigroups of embedding dimension $2$.

\begin{theorem}\label{d=2}
If $a,b,d$ are relatively prime positive integers, then
$$g\left(\langle a,b \rangle/d\right)=\frac{1}{2d}(a-1)(b-1)+C_{a,b,d},$$ where $C_{a,b,d}$ is a constant dependent only on $d$ and the residue classes of $a$ and $b$ modulo $d$.
\end{theorem}
\begin{proof}
Let $S=\langle a,b\rangle$. Sylvester's theorem (Theorem \ref{SylvesterTheorem}) tells us that $g(S)=(a-1)(b-1)/2$. Invoking \eqref{Eq2} yields
\[g\left(S/d\right)=A_0=\frac{1}{d}\left(g(S)+\sum_{j=1}^{d-1}\frac{1-P_S(\zeta_d^j)}{1-\zeta_d^j}\right)=\frac{1}{2d}(a-1)(b-1)+\frac{1}{d}\sum_{j=1}^{d-1}\frac{1-P_S(\zeta_d^j)}{1-\zeta_d^j}.\qedhere\]
\end{proof}

As mentioned in the introduction, the following corollary generalizes a result of Strazzanti \cite{Stra}.  

\begin{corollary}\label{corofmain}
Fix a positive integer $d$. For relatively prime positive integers $a$ and $k$, the genus $g\left(\langle a,a+k\rangle/d\right)$ is a quasipolynomial in $a$ of degree $2$ with leading coefficient $\frac{1}{2d}$.
\end{corollary}


%
%
%
%
%
%
%

\section{Numerical Semigroups Generated by Arithmetic Progressions}\label{sec:arithmetic}

Two well-studied types of numerical semigroups are those of the form $\langle a,a+k,a+2k\rangle$ and $\langle a,a+k,\ldots,a+(a-1)k\rangle$, where $a$ and $k$ are relatively prime positive integers. Note that the latter numerical semigroup is, in fact, the numerical semigroup generated by the terms of the infinite arithmetic progression $a,a+k,a+2k,\ldots$. We will make use of the following special case of \cite[Corollary 3.2]{Num}.

\begin{prop}[\cite{Num}]\label{Num}
If $a$ and $k$ are relatively prime positive integers, then the numerical semigroup $\langle a,a+k,a+2k \rangle$ is symmetric if and only if $a$ is even.
\end{prop}

\begin{theorem}\label{even}
Let $d\geq 3$ be an integer. Let $a$ and $k$ be relatively prime positive integers. If $a=sd$ is a positive even multiple of $d$, then $\langle a, a+k, a+2k\rangle/d$ is symmetric. 
\end{theorem}
\begin{proof}
Let $S=\langle a,a+k,a+2k\rangle$. We know that $\gcd(d,k)=1$ because $a$ and $k$ are relatively prime. An element of $S$ has the form $ax+(a+k)y+(a+2k)z$ for nonnegative integers $x,y,z$. In order for this element to be a multiple of $d$, we need $d\mid(y+2z)$. It follows that $S/d$ is generated by $s,a+k,a+2k$ and all integers of the form $\frac{(a+k)y+(a+2k)z}{d}$ such that $d\mid(y+2z)$ and $y,z\in\{0,1,\ldots,d-1\}$.

First, assume $d$ is odd, say $d=2t+1$. By hypothesis, $s$ is even. The pairs $(y,z)$ such that $d\mid(y+2z)$ and $y,z\in\{0,1,\ldots,d-1\}$ are \[(d-2,1),(d-4,2),\ldots,(1,t),(d-1,t+1),\ldots,(2,2t).\] The corresponding values of $\frac{(a+k)y+(a+2k)z}{d}$ are \[a+k-s,a+k-2s,\ldots,a+k-ts,a+2k+ts,\ldots,a+2k+s.\] As a result, $S/d=\langle s,a+k-ts,a+2k\rangle$. But $a=sd=(2t+1)s$, so $S/d=\langle s,s+k+ts,s+2k+2ts\rangle$. By Proposition~\ref{Num}, $S/d$ is symmetric. 

Next, assume $d$ is even, say $d=2t$. The pairs $(y,z)$ such that $d\mid(y+2z)$ and $y,z\in\{0,1,\ldots,d-1\}$ are $(d-2,1),(d-4,2),\ldots,(0,t),(d-2,t+1),\ldots,(2,2t-1)$. The corresponding values of $\frac{(a+k)y+(a+2k)z}{d}$ are \[a+k-s,a+k-2s,\ldots,a+k-ts,a+2k+ts,\ldots,a+2k+s.\] It follows that $S/d=\langle s,a+k-st,a+2k\rangle=\langle s,k+st,2k+2st\rangle=\langle s,k+st\rangle$. It is well-known that every numerical semigroup with embedding dimension $2$ is symmetric, so the proof is complete.  
\end{proof}

\begin{corollary}
Let $d\geq 4$ be an even integer. Let $a$ and $k$ be relatively prime positive integers. If $a=sd$, then $F(\langle a, a+k, a+2k\rangle/d)=(s-1)(a+2k)/2-s$ and $g(\langle a, a+k, a+2k\rangle/d)=(s-1)(a+2k-2)/4$.
\end{corollary}
\begin{proof}
Write $d=2t$. We know from the proof of Theorem \ref{even} that $\langle a, a+k, a+2k\rangle/d = \langle s, k+st \rangle$. In general, it is known that $F(\langle x, y \rangle)=xy-x-y$ and $g(\langle x,y \rangle)=(x-1)(y-1)/2$. Setting $x=s$ and $y=k+st$ gives the desired result. 
\end{proof}

Let $n$ be a nonzero element of a numerical semigroup $S$. The \emph{Ap\'{e}ry set of $n$ in $S$} is the set
\[\Ap( S , n ) = \{ s \in S:s-n \notin S \} .\] The following fundamental results which we will need in the proof of Theorem \ref{cor:aoddmultipleofd}, are known as \emph{Selmer's formulas}. 

\begin{theorem} [\cite{assi}, Proposition 12] \label{selmer}
If $S$ is a numerical semigroup, then
\begin{enumerate}
\item $F(S)=\max(\Ap(S,m))-m$.
\item $g(S)=\frac{1}{m}\sum\limits_{w\in \Ap(S,m)}w-\frac{m-1}{2}$.
\end{enumerate}
\end{theorem}

\begin{theorem}
Let $S=\langle a,a+k,a+2k\rangle$, where $a$ and $k$ are relatively prime positive integers. Suppose $d$ is a positive divisor of $a$, say $a=sd$. If $a$ is odd, then 
\[F(S/d)=\left((s-1)t+\frac{s-1}{2}\right)s+(s-1)k-s,\] and 
\[
g(S/d)=\frac 12\left(s(s-1)t+\frac{s^2-1}{2}+(s-1)k-(s-1)\right).
\]
\end{theorem}
\begin{proof}
From the proof of Theorem \ref{even}, we have $S/d=\langle s,s+k+ts,s+2k+2ts \rangle = \langle s,(t+1)s+k,(2t+1)s+2k \rangle$, where $d=2t+1$. Now, \[\Ap(S,s)=\{0,(t+1)s+k, (2t+1)s+2k, (3t+2)s+3k, (4t+2)s+4k, \ldots,\] \[ \left((s-2)t+\tfrac{s-1}{2}\right)s+(s-2)k ,\left((s-1)t+\tfrac{s-1}{2}\right)s+(s-1)k\}\] since $\gcd(s,k)=1$. The desired result is now immediate by Theorem \ref{selmer}. 
\end{proof}

\begin{corollary}\label{cor:aoddmultipleofd}
Let $S=\langle a,a+k,a+2k\rangle$, where $a$ and $k$ are relatively prime positive integers. Suppose $d$ is a positive divisor of $a$, say $a=sd$. If $a$ is odd, then $2g(S/d)-F(S/d)=\dfrac{s+1}{2}$.
\end{corollary}


We now consider the numerical semigroup generated by the terms of an infinite arithmetic progression $a,a+k,a+2k,\ldots$, where $a$ and $k$ are relatively prime positive integers. Equivalently, this is the numerical semigroup \[S=\langle a,a+k,a+2k,\ldots,a+(a-1)k\rangle.\]  

\begin{prop} \label{quotient}
Let $S=\langle a,a+k,a+2k,\ldots,a+(a-1)k\rangle$, where $a$ and $k$ are relatively prime positive integers. Let $d$ be a positive divisor of $a$, say $a=sd$. We have $\displaystyle g(S/d)=\frac{F(S/d)+s-1}{2}$.
\end{prop}
\begin{proof}
We first reduce to the case $d=1$. Let $x_{i,a}$ denote the unique element of the set $\{0,1,\ldots,a-1\}$ such that $x_{i,a}k\equiv i\pmod a$ (the uniqueness comes from the assumption that $\gcd(a,k)=1$). We claim that the elements of $S$ that are congruent to $i$ modulo $a$ are precisely the integers of the form $ta+x_{i,a}k$ for positive integers $t$. To see this, notice that every nonnegative integer linear combination of generators of $S$ has the form $ta+mk$; in order to have $ta+mk\equiv i\pmod a$, we need $m\equiv x_{i,a}\pmod a$. Now, every element in $S/d$ is of the form $(ta+x_{di,a})/d$ for $i\in\{0,1,\ldots,s\}$. Since $x_{di,a}/d=x_{i,s}$, we have $S/d=\langle s,s+k,s+2k,\ldots,s+(s-1)k \rangle$. This shows that it suffices to prove the desired result in the case in which $d=1$ and $s=a$. However, this follows immediately from Selmer's formula for the genus (Theorem \ref{selmer}) along with Proposition 2.6 and Theorem 2.8 in \cite{Omidali}. Specifically, with $d=1$ and $s=a$, we have \[g(S/d)=g(S)=\frac{(k+1)(a-1)}{2}\quad\text{and}\quad F(S/d)=F(S)=k(a-1). \qedhere\]
\end{proof}

Some useful facts concerning the numerical semigroup $S$ follow from the Proposition \ref{quotient}. It is easy to see that
\[S=\{as+kt:s\geq1,t\geq0\}=\langle a,k\rangle\setminus\{k,2k,\ldots,(a-1)k\}.\]
Suppose $d\mid(a-1)k$. Since $F(\langle a,k\rangle)=ak-a-k<(a-1)k$, we know that $F(S/d)=(a-1)k/d$. Furthermore, $g(S/d)$ is equal to the sum of $g(\langle a,k\rangle/d)$ and the number of multiples of $d$ in $\{k,2k,\ldots,(a-1)k\}$. Thus, we obtain the following result.

\begin{prop}
Let $S=\langle a,a+k,a+2k,\ldots,a+(a-1)k\rangle$, where $a$ and $k$ are relatively prime positive integers. If $d\mid k$, then $g(S/d)=\dfrac{F(S/d)+a-1}{2}$.
\end{prop}
\begin{proof}
This follows from the above analysis. Specifically, when $d\mid k$, we have $F(S/d)=(a-1)k/d$. It is also straightforward to check that $\langle a,k\rangle/d=\langle a,k/d\rangle$. Thus, 
\[
g(S/d)=g(\langle a,k\rangle/d)+(a-1)=g(\langle a,k/d\rangle)+(a-1)=\frac{(a-1)(k/d+1)}{2}\] \[=\frac{F(S/d)+a-1}{2}.\qedhere
\]
\end{proof}

The results in this section lead naturally to the following problem. 

\begin{problem}
Let $S =\langle a,a+k,a+2k,\dots,a+\ell k \rangle$, where $a$ and $k$ are relatively prime positive integers and $3 \leq \ell  \leq a-2$. Find $g(S/d)$ and $F(S/d)$, as well as relationships between these quantities, for positive integers $d$.
\end{problem}

\section{Acknowledgments}
A. Adeniran, C. Defant, Y. Gao, C. Hettle, Q. Liang, H. Nam, and A. Volk were partially supported by NSF-DMS grant \#1603823 "Collaborative Research: Rocky Mountain Great Plains Graduate Research Workshops in Combinatorics" and by NSF-DMS grant \#1604458, "Collaborative Research: Rocky Mountain Great Plains Graduate Research Workshops in Combinatorics."
S. Butler and P. E. Harris were partially supported by NSA grant \#H98230-18-1-0017, "The 2018 and 2019 Rocky Mountain -- Great Plains Graduate Research Workshops in Combinatorics." C. Defant was also supported by a Fannie and John Hertz Foundation Fellowship and an NSF Graduate Research Fellowship.


\begin{thebibliography}{99}

\bibitem{assi}
A. Assi and P. A. Garc\'{i}a-S\'{a}nchez,
\textit{Numerical semigroups and applications}.
RSME Springer Series, 1. Springer, [Cham], 2016. xiv+106


\bibitem{Davison}
J. L. Davison,
On the linear Diophantine problem of Frobenius, \emph{J. Number Theory} {\bf 48} (1994), 353--363.

\bibitem{del}
M. Delgado, P. A. Garc\'{i}a-S\'{a}nchez, and J. C. Rosales, Numerical semigroups problem list. arXiv: 1304.6552.

%

\bibitem{Ze}
Ze Gu, Xilin Tang,
The doubles of one half of a numerical semigroup.
J. Number Theory {\bf 163} (2016), 375--384.

\bibitem{Kap}
N. Kaplan,
Counting numerical semigroups.
\emph{Amer. Math. Monthly} {\bf 124} (2017), 862–-875.


\bibitem{MSS}
R. Mehta, J. Saha, and I. Sengupta,
Numerical semigroups generated by concatenation of two arithmetic sequences. arXiv: 1802.02564.

\bibitem{Mos1} 
A. Moscariello,
Generators of a fraction of a numerical semigroup.
arXiv: 1402.4905.

\bibitem{Mosc1}
A. Moscariello, The first elements of the quotient of a numerical semigroup by a positive integer. \emph{Semigroup Forum} {\bf 90} (2015), 126--134.

\bibitem{Mosc2}
A. Moscariello,
Minimal relations and the Diophantine Frobenius problem in embedding dimension three. arXiv: 1608.06137.

\bibitem{Num}
T. Numata, Numerical semigroups generated by arithmetic sequences.
\emph{Proceedings of the Institute of Natural Sciences, Nihon University} \textbf{49} (2014).

\bibitem{Omidali}
M. Omidali and F. Rahmati, On the type and the minimal presentation of
certain numerical semigroups. \emph{Comm. Algebra} {\bf 37} (2009), 1275--1283. 

\bibitem{RR}
A. M. Robles and J. C. Rosales, Equivalent proportionally modular Diophantine inequalities. \emph{Archiv Math.} {\bf 90} (2008), 24--30. 


\bibitem{RU}
J. C. Rosales and J. M. Urbano, Proportionally modular Diophantine inequalities and full semigroups, \emph{Semigroup Forum} {\bf 72} (2006), 362--374.

\bibitem{RosGar}
J. C. Rosales and P. Garc\'{i}a-S\'{a}nchez, \emph{The quotient of a numerical semigroup by a positive integer}. In: Numerical Semigroups. Developments in Mathematics (Diophantine Approximation: Festschrift for Wolfgang Schmidt), Vol 20. Springer, New York, NY.

\bibitem{Stra}
F. Strazzanti, Minimal genus of a multiple and Frobenius number of a quotient of a numerical semigroup.
\emph{Internat. J. Algebra Comput.} {\bf 25} (2015), 1043--1053.



\bibitem{Syl}
J. J. Sylvester, Excursus on rational fractions and partitions. \emph{Amer. J. Math.} {\bf 5} (1882), 111--136.



\end{thebibliography}
\end{document}